\documentclass{article}
\usepackage[utf8]{inputenc}
\usepackage{fullpage}
\usepackage{amsmath}
\usepackage{amsfonts}
\usepackage{amssymb}
\usepackage{tikz-cd}
\usepackage{array}
\usepackage{float}
\usepackage{soul}
\usepackage{zref-savepos}
\usepackage{amsthm}
\usepackage{multirow, bigstrut}
\usepackage[font={small,it}]{caption}

\usetikzlibrary{calc}
\usepackage{comment}
\usepackage{enumerate}
\usepackage{tikz}
\usetikzlibrary{shapes.geometric}

\newtheorem{theorem}{Theorem}[section]

\newtheorem{lemma}[theorem]{Lemma}
\newtheorem{proposition}[theorem]{Proposition}
\newtheorem{corollary}[theorem]{Corollary}

\theoremstyle{definition}
\newtheorem{definition}{Definition}[section]
\newtheorem{remark}[definition]{Remark}

\newcommand{\C}{\mathbb{C}}
\newcommand{\R}{\mathbb{R}}

\newcommand{\K}{\mathbb{K}}

\title{Generalizations of the Matching Polynomial to the Multivariate Independence Polynomial}
\author{Jonathan Leake \\ Nick Ryder
\footnote{This research was supported by NSF Grant CCF-155375} \\}
\begin{document}

\maketitle

\begin{abstract}
We generalize two main theorems of matching polynomials of undirected simple graphs, namely, real-rootedness and the Heilmann-Lieb root bound. Viewing the matching polynomial of a graph $G$ as the independence polynomial of the line graph of $G$, we determine conditions for the extension of these theorems to the independence polynomial of any graph. In particular, we show that a stability-like property of the multivariate independence polynomial characterizes claw-freeness. Finally, we give and extend multivariate versions of Godsil's theorems on the divisibility of matching polynomials of trees related to $G$.
\end{abstract}

\section{Introduction}
Given a graph $G = (V,E)$, the matching polynomial of $G$ and the independence polynomial of $G$ are defined as follows.
\[
\mu(G) := \sum_{\substack{M \subset E \\ M,\text{matching}}} \left(-x^2\right)^{|M|} ~~~~~~~~~~~ I(G) := \sum_{\substack{S \subset V \\ S,\text{independent}}} x^{|S|}
\]
The real-rootedness of the matching polynomial and the Heilmann-Lieb root bound are important results in the theory of undirected simple graphs. In particular, real-rootedness implies log-concavity and unimodality of the matchings of a graph, and recently in \cite{ramanujan} the root bound was used to show the existence of Ramanujan graphs. Additionally, it is well-known that the matching polynomial of a graph $G$ is equal to the independence polynomial of the line graph of $G$. With this, one obtains the same results for the independence polynomials of line graphs. This then leads to a natural question: what properties extend to the independence polynomials of all graphs?

Generalization of these results to the independence polynomial has been partially successful. About a decade ago, Chudnovsky and Seymour \cite{CS} established the real-rootedness of the independence polynomial for claw-free graphs. (The independence polynomial of the claw is not real-rooted.) A general root bound for the independence polynomial was also given by \cite{rootbounds}, though it is weaker than that of Heilmann and Lieb. As with the original results, these generalizations are proven using univariate polynomial techniques.

More recently, Borcea and Br{\"a}nd{\'e}n used their characterization of stability-preserving operators \cite{bb1}, \cite{bb2} to give a simple and intuitive proof of the real-rootedness of the matching polynomial. To that end, they use the following multivariate generalization of real rootedness.

% (1/10): rearranged defn a bit
\begin{definition}
Let $\mathcal{H}_+$ denote the open upper half-plane of $\C$. A polynomial $p \in \C[z_1,...,z_n]$ is said to be \emph{stable} if either $(z_1,...,z_n) \in \mathcal{H}_+^n$ implies $p(z_1,...,z_n) \neq 0$. Additionally, if $p \in \R[z_1,...,z_n]$, we say that $p$ is \emph{real stable}. (For simplicity, we also say that the zero polynomial is both stable and real stable.) Notice that a univariate polynomial is real stable iff it is real-rooted.
\end{definition}

Borcea and Br{\"a}nd{\'e}n then prove something much stronger than real-rootedness: they actually show that the \emph{multivariate} matching polynomial is real stable. Beyond its surprising simplicity, their proof also suggests that the multivariate approach may be the more natural one. That said, the first part of this paper is a partial generalization of this stability result to the multivariate independence polynomial of claw-free graphs. In particular, we prove a result related to the real-rootedness of certain weighted independence polynomials. This result was originally proven by Engstr\"om in \cite{weighted} by bootstrapping the Chudnovsky and Seymour result for rational weights and using density arguments. The proof we give here is completely self contained and implies both the original Chudnovsky and Seymour result as well as the weighted generalization. By using a multivariate framework to directly prove the more general result, we obtain a simple inductive proof which we believe better captures the underlying structure.
 
% More general stability properties of the multivariate independence polynomial have been studied in connection to the Lov{\'a}sz local lemma, in \cite{lll}.

In addition, the full importance of the claw (3-star) graph is not immediately clear from the univariate framework. Since the result of Chudnovsky and Seymour, there have been attempts to explain more conceptually why the claw-free premise is needed for real-rootedness. In particular, some graphs containing claws actually have real-rooted independence polynomials, disproving the converse to the univariate result. On the other hand, the stronger stability-like property we use here turns out to be equivalent to claw-freeness, yielding a satisfactory converse.

In the second part of this paper, we then extend the Heilmann-Lieb root bound by generalizing some of Godsil's work on the matching polynomial. In \cite{godsil}, Godsil demonstrated the real-rootedness of the matching polynomial of a graph by showing that it divides the matching polynomial of a related tree. (For a tree, root properties are more easily derived.) We prove a similar result for the multivariate matching polynomial, and then we determine conditions for which these divisibility results extend to the multivariate independence polynomial. Further, we prove the Heilmann-Lieb root bound for the independence polynomial of a certain subclass of claw-free graphs. By considering a particular graph called the Schläfli graph, we demonstrate that this root bound does not hold for all claw-free graphs and provide a weaker bound in the general claw-free case.

% \subsubsection*{Roadmap} We apply our generalized notion of a stable polynomial to the independence polynomial in the following section, ultimately yielding a new proof the independence polynomial is real rooted which does not require the intricate recursive structure of simplicial cliques found in the original paper \cite{CS}. 

% The remainder of our paper extends the classical root bounds from \cite{HL} to the independence polynomial. To do this we use the divisibility relations of path trees originally discussed by Godsil and found in his book \cite{godsil}. 

% (1/10): completely reorganized this section
\section{Stability Theory}
Before the graph theoretic results, we give a bit of background on stability theory. We then generalize the typical notion of stability in a way that gives a natural extension of the matching polynomial stability result.

% (11/12): added defn for \partial_{z_k}
In what follows, let $\mathcal{H}_+$ denote the open upper half-plane of $\C$, let $\R_+$ denote the nonnegative real numbers, and let $\K$ denote a field, either $\R$ or $\C$. For $p \in \K[z_1,...,z_n]$ and $\mathbf{t} = (t_1,...,t_n) \in \K^n$, define $p(\mathbf{t}z) := p(t_1z,t_2z,...,t_nz)$, which is a univariate polynomial. For $\mathbf{t} = (t_1,...,t_n) \in \K^n$ and $k \in [n] := \{1,...,n\}$, let $(t_1,...,\hat{t}_k,...,t_n)$ denote the vector in $\K^{n-1}$ which is the vector $\mathbf{t}$ with the $k^\text{th}$ element removed. Also, for all $k$ we use the shorthand $\partial_{z_k} := \frac{\partial}{\partial z_k}$.

\subsection{Interlacing}

The notion of interlacing polynomials is intimately related to the theory of stable polynomials. That said, we now define this notion and state a few of its important properties.

% (11/12): updated definition to account for sign
\begin{definition}
Let $p,q \in \R[z]$ be real-rooted polynomials given by $p(z) = C_1 \prod_{k=1}^n (z - \lambda_k)$ and $q(z) = C_2 \prod_{k=1}^m (z - \gamma_k)$, where $n$ and $m$ differ by at most 1 and $m \leq n$. We write $q \ll p$, or say $q$ \emph{interlaces} $p$, if $\lambda_1 \geq \gamma_1 \geq \lambda_2 \geq \gamma_2 \geq \cdots$ and $C_1 \cdot C_2 > 0$. If the roots alternate in the same way but $C_1 \cdot C_2 < 0$, we swap the order of $p$ and $q$ in this relation.
\end{definition}

% (11/12): use of pos leading coeff
\begin{definition}
Let $p_1,p_2,...,p_m \in \R[z]$ be real-rooted polynomials with positive leading coefficients. We say that $p_1,p_2,...,p_m$ \emph{have a common interlacing}, if there exists $f \in \R[z]$ such that $f \ll p_k$ for all $k \in [m]$.
\end{definition}

Notice in the above definition that the connotation of ``$\ll$'' as an order symbol presents itself in the fact that the ``larger'' polynomial has a larger maximum root (when $C_1 \cdot C_2 > 0$). However, $\ll$ is not a partial order.

The next result gives a link between the concept of interlacing and the roots of linear combinations of polynomials. It is typically attributed to Obreshkoff, but can be viewed as a reformulation of the Hermite-Biehler theorem.

\begin{proposition}[Obreshkoff's Theorem]
For $p,q \in \R[z]$ with real roots, $\alpha p + \beta q$ is real-rooted for all $\alpha,\beta \in \R$ iff $p \ll q$ or $q \ll p$.
\end{proposition}

To generalize Obreshkoff's Theorem to many polynomials, Chudnovsky and Seymour make the following definition and prove the following equivalence.

% In this direction, Chudnovsky and Seymour define:
\begin{definition}
We say that $p_1, \ldots, p_m \in \R[z]$ are \emph{compatible} if all convex combinations are real rooted.
\end{definition}
% We now give a similar result which will be more important in what follows. Notice that the equivalent conditions of the following theorem are how Chudnovsky and Seymour define ``compatible polynomials'' in \cite{CS}.

% % (11/12): updated to say pos leading coeff
% \begin{proposition} \label{thm:cs_compat}
% Fix $p_1,...,p_m \in \R[z]$ with all real roots and positive leading coefficients. We have that $\sum_{k=1}^m \alpha_k p_k$ is real-rooted for all $\alpha_k \in \R_+$ iff every convex combination of $p_1,...,p_m$ is real-rooted iff $p_1,...,p_m$ have a common interlacing.
% \end{proposition}

% In Chudnovsky and Seymour's original independence polynomial paper \cite{CS} they establish equivalences between these definitions for polynomials with positive leading coefficients.

\begin{theorem}[\cite{CS}]\label{thm:cs_compat}
Let $p_1, \ldots, p_k \in \R[z]$ be polynomials with positive leading coefficients. The following are equivalent.
\begin{enumerate}
    \item $p_i$ and $p_j$ are compatible for all $i \neq j$.
    \item $p_i$ and $p_j$ have a common interlacing for all $i \neq j$.
    \item $p_1,...,p_k$ are compatible.
    \item $p_1,...,p_k$ have a common interlacing.
\end{enumerate}
\end{theorem}

% % (11/12): updated reference and use of pos leading coeff
% Note that (assuming all polynomials have positive leading coefficient) this definition is equivalent to the property that each pair of $p_1,...,p_m$ has a common interlacing (e.g., see \cite{CS}).

\subsection{Real Stability}
We now give a condition which is equivalent to the notion of stability defined above.

% (1/10): added citation
\begin{proposition}[\cite{bb1}, Lemma 1.5] \label{thm:stab_equiv}
A polynomial $p \in \K[z_1,...,z_n]$ for $\K = \C$ (resp. $\K = \R$) is stable (resp. real stable) iff for every $\mathbf{t} \in \R_+^n$ and every $\mathbf{y} \in \R^n$, the univariate restriction $p(\mathbf{t}z + \mathbf{y})$ is stable. Note that if $\K = \R$, the univariate restrictions will be real-rooted.
\end{proposition}

We give this equivalent condition to emphasize the sense in which we generalize stability in the next section. As will be seen, this generalization turns out to work well with the multivariate independence polynomial. Before this though, we give a bit more stability theory. The following will also serve as a base for generalization in the next section.

\begin{proposition}[Closure Properties]
Let $p,q \in \K[z_1,...,z_n]$ be stable (resp. real stable) polynomials, and fix $k \in [n]$. Then the following are also stable (resp. real stable).
\begin{enumerate}[(i)]
    \item $p \cdot q$ (product)
    \item $\partial_{z_k}p$ (differentiation)
    \item $z_k\partial_{z_k}p$ (degree-preserving differentiation)
    \item $p(z_1,...,z_{k-1},r,z_{k+1},...,z_n)$, for $r \in \R$ (real specialization)
    \item $p(z_1,...,z_{k-1},z_1,z_{k+1},...,z_n)$ (projection)
    \item $z_k^{\text{deg}_k(p)} p(z_1,...,z_{k-1},-z_k^{-1},z_{k+1},...,z_n)$ (inversion)
\end{enumerate}
Here, $\text{deg}_k(p)$ is the degree of $z_k$ in $p$.
\end{proposition}

The next result is a stability equivalence theorem of Borcea and Br{\"a}nd{\'e}n, which is essentially a generalization of the Hermite-Biehler theorem. It is the inspiration for the main theorem of the next section.

% (1/10): added citation
\begin{theorem}[\cite{bb1}, Lemma 1.8] \label{thm:alg_hb}
For $p,q \in \R[z_1,...,z_n]$, $p + z_{n+1}q$ is real stable iff for every $\mathbf{t} \in \R_+^n$ and every $\mathbf{y} \in \R^n$, we have that $q(\mathbf{t}z + \mathbf{y})$ interlaces $p(\mathbf{t}z + \mathbf{y})$.
\end{theorem}

Finally, we give an equivalent condition for real stability of multi-affine polynomials. This will be a useful result for demonstrating counterexamples to real-rootedness.

\begin{definition}
A polynomial $p \in \K[z_1,...,z_n]$ is said to be \emph{multi-affine} if it is of degree at most one in each variable.
\end{definition}

\begin{proposition}[\cite{strongrayleigh}]
A multi-affine polynomial $p \in \R[z_1,...,z_n]$ is real stable iff $p$ is strongly Rayleigh. That is, iff for every $j \neq k \in [n]$ and every $\mathbf{x} \in \R^n$, we have the following.
\[
(\partial_{z_j}p)(\mathbf{x}) \cdot (\partial_{z_k}p)(\mathbf{x}) \geq (\partial_{z_j}\partial_{z_k}p)(\mathbf{x}) \cdot p(\mathbf{x})
\]
\end{proposition}

\subsection{Same-phase Stability}

We now introduce a new notion of stability. Notice that the connection between the following conditions is similar to that which is given by Proposition \ref{thm:stab_equiv}.

\begin{definition}
A polynomial $p \in \R[z_1,...,z_n]$ is said to be \emph{same-phase stable} if one of the following equivalent conditions is satisfied.
\begin{enumerate}[(i)]
    \item For every $\mathbf{t} \in \R_+^n$, the univariate restriction $p(\mathbf{t}z)$ is stable (and therefore real rooted).
    \item If $\arg(z_1) = \arg(z_2) = \cdots = \arg(z_n)$, then $p(z_1,...,z_n) = 0$ implies $z_k \not\in \mathcal{H}_+$ for some $k$.
\end{enumerate}
We will primarily make use of condition (i).
\end{definition}

This notion is strictly weaker than that of ``stable'', and it will serve as the basic concept in what follows (as stability and real stability did in the previous section). Next, we define a notion of compatibility for real same-phase stable polynomials, which is similar to that of Chudnovsky and Seymour in \cite{CS}.
%Note that here we define compatibility in terms of interlacing, rather than in terms of the real-rootedness of convex combinations. The equivalences of Theorem \ref{thm:cs_compat} show that this difference is inconsequential.

% (12/21): use of nonneg coeff
\begin{definition}
Polynomials $p_1,...,p_m \in \R_+[z_1,...,z_n]$ with nonnegative coefficients are said to be \emph{same-phase compatible} if $p_k$ is same-phase stable for all $k$, and the polynomials $\{p_k(\mathbf{t}z)\}_{k=1}^m$ are compatible for each $\mathbf{t} \in \R_+^n$. Note that by Theorem \ref{thm:cs_compat}, we could instead require $\{p_k(\mathbf{t}z)\}_{k=1}^m$ have a common interlacing for each $\mathbf{t} \in \R_+^n$.
\end{definition}

% (1/11): clarifying remark
\begin{remark}
In order to utilize the theory of interlacing and compatible polynomials, we need to assume that the polynomials we are using have nonnegative coefficients. This is because results like Theorem \ref{thm:cs_compat} no longer hold if negative or complex coefficients are allowed. That said, this restriction is not required to define same-phase stable polynomials, and many other properties also hold without it.
\end{remark}

We now can apply Chudnovsky and Seymour's equivalence result (Theorem \ref{thm:cs_compat}) to get the following:

\begin{corollary} \label{cor:pw_samephase_compat}
Let $p_1, \ldots, p_k \in \R_+[z_1,...,z_n]$ be polynomials with nonnegative coefficients. The following are equivalent.
\begin{enumerate}
    \item $p_i$ and $p_j$ are same-phase compatible for all $i \neq j$.
    \item $p_1,...,p_k$ are same-phase compatible.
\end{enumerate}
\end{corollary}

% \begin{remark} \label{rem:pw_samephase_compat}
% Applying \ref{thm:cs_compat} to $p_i(\mathbf{t}z)$, we have that $p_1,...,p_m$ are same-phase compatible iff they are pairwise same-phase compatible.
% \end{remark}

% To see this as a generalization of the definition found in \cite{bb1}, notice by Proposition \ref{thm:cs_compat} that $\{p_k(\mathbf{t}z)\}_{k=1}^m$ have a common interlacing iff $\sum \alpha_k p_k(\mathbf{t}z)$ is real-rooted for all $\alpha_k \in \R_+$. This leads to the following characterization of same-phase compatibility.

% \begin{corollary}
% Polynomials $p_1,...,p_m \in \R[z_1,...,z_n]$ are same-phase compatible iff every convex combination of $p_1,...,p_m$ is same-phase stable.
% \end{corollary}

\subsection{Same-phase Stability for Multi-affine Polynomials}

We now begin to develop a general theory of same-phase stability for multi-affine real polynomials. This class of polynomials is of particular importance here, as most multivariate graph polynomials are real and multi-affine. We start by giving some basic closure properties.

% (1/11): added extra line about nonnegative coefficients
\begin{proposition}[Closure Properties]
Let $p \in \R[z_1,...,z_n]$ and $q \in \R[w_1,...,w_m]$ be multi-affine same-phase stable polynomials, and fix $k \in [n]$. Then the following are also multi-affine same-phase stable. Note that if in addition $p$ and $q$ have nonnegative coefficients, then the following do as well.
\begin{enumerate}[(i)]
    \item $p \cdot q$ (disjoint product)
    \item $\partial_{z_k}p$ (differentiation)
    \item $z_k\partial_{z_k}p$ (variable selection)
    \item $p(z_1,...,z_{k-1},0,z_{k+1},...,z_n)$ (variable deselection)
    \item $z_1z_2 \cdots z_n p(z_1^{-1},...,z_n^{-1})$ (selection inversion)
\end{enumerate}
\end{proposition}
\begin{proof}
$(i)$ Straightforward.

% (12/21): referenced Hurwitz's theorem instead of the vague ``continuity''
% (1/11): fixed proofs of (iv) and (v) to deal with nonnegative reals rather than strictly positive
$(ii)$ Fix $\mathbf{t} \in \R_+^n$, letting $t_k$ vary. Also, define $\mathbf{t}_0 := (t_1,...,\hat{t}_k,...,t_n)$. So, $p(\mathbf{t}z)$ is real-rooted for any $t_k \in \R_+$. By Hurwitz's theorem,
\[
(\partial_{z_k}p)(\mathbf{t}_0z) = \lim_{t_k \rightarrow \infty} t_k^{-1}p(\mathbf{t}z)
\]
is also real-rooted. So, $\partial_{z_k}p$ is same-phase stable.

$(iii)$ This follows from $(i)$, since $(z_k\partial_{z_k}p)(\mathbf{t}z) = t_kz(\partial_{z_k}p)(\mathbf{t}_0z)$ is real-rooted iff $(\partial_{z_k}p)(\mathbf{t}_0z)$ is.

% $(iv)$ Fix $\mathbf{t} \in \R_+^n$, letting $t_k$ vary. Using the argument from $(i)$, $p(\mathbf{t}z)$ is real-rooted and
% \[
% p(t_1z,...,t_{k-1}z,0,t_{k+1}z,...,t_nz) = \lim_{t_k \rightarrow 0^+} p(\mathbf{t}z)
% \]
% is also real-rooted. So, $p(t_1z,...,t_{k-1}z,0,t_{k+1}z,...,t_nz)$ is same-phase stable.
$(iv)$ For any $\mathbf{t} \in \R_+^n$ with $t_k = 0$, we have that $p(t_1z,...,t_{k-1}z,0,t_{k+1}z,...,t_nz) = p(\mathbf{t}z)$ is real-rooted by definition of same-phase stability.

% $(v)$ Given $\mathbf{t} \in \R_+^n$, we have that $p(\mathbf{t}^{-1}z)$ has real roots, say at $\gamma_1, \ldots, \gamma_m$. So, $z^np(\mathbf{t}^{-1}z^{-1}) = t_1 z \ldots t_n z \cdot p((t_1z)^{-1}, \ldots, (t_nz)^{-1})$ has real roots at $\gamma_1^{-1}, \ldots, \gamma_m^{-1}$. Of course, some of these inverse zeros may be missing when some $\gamma_j = 0$, and there may be extra zeros at $z = 0$. However, this will not affect the real-rootedness of the ``inverted'' polynomial.
$(v)$ Given $\mathbf{t} \in \R_+^n$ with strictly positive entries, we have that $p(\mathbf{t}^{-1}z)$ has real roots, say at $\gamma_1, \ldots, \gamma_m$. So, $z^np(\mathbf{t}^{-1}z^{-1}) = t_1 z \ldots t_n z \cdot p((t_1z)^{-1}, \ldots, (t_nz)^{-1})$ has real roots at $\gamma_1^{-1}, \ldots, \gamma_m^{-1}$. Of course, some of these inverse zeros may be missing when some $\gamma_j = 0$, and there may be extra zeros at $z = 0$. However, this will not affect the real-rootedness of the inverted polynomial. Hurwitz's theorem then allows us to limit to all $\mathbf{t}\in \R_+^n$.

\end{proof}

The names given to some of the closure properties are specific to multi-affine polynomials. In particular, ``variable selection'' (resp. ``variable deselection'') refers to the fact that the associated actions will pick out the terms of $p$ which contain (resp. do not contain) a particular variable. Then, ``selection inversion'' inverts which terms contain which variables. The idea here is to give a combinatorial interpretation to these actions. For example, if the variables correspond to vertices on some graph, then variable deselection might correspond to removal of some vertex.

The next definition is inspired by $p + z_{n+1}q$ used in Theorem \ref{thm:alg_hb}. The proposition that follows then relates this definition to multi-affine polynomials.

\begin{definition}
Let $p,f_0,f_1,...,f_m \in \R[z_1,...,z_n]$ be polynomials, not necessarily multi-affine, such that
\[
p = f_0 + z_{i_1} f_1 + \cdots + z_{i_m} f_m.
\]
We call such an expression a \emph{proper splitting} of $p$ (with respect to $\{z_{i_j}\}_j$) if none of the $f_k$'s depend on any of the $z_{i_j}$'s. We also say that $\{z_{i_j}\}_{j=1}^m$ \emph{splits} $p$.
\end{definition}

\begin{proposition}
Let $p \in \K[z_1,...,z_n]$ be a multi-affine polynomial, and suppose $\{z_{i_j}\}_{j=1}^m$ splits $p$. Then $p$ has a unique proper splitting with respect to $\{z_{i_j}\}_j$, expressed as
\[
p = p_0 + \sum_{j=1}^m z_{i_j} \partial_{z_{i_j}} p,
\]
where $p_0$ is the polynomial $p$ with the variables $\{z_{i_j}\}_j$ evaluated at 0.
\end{proposition}

Another way to think about this proposition is as follows. For a multi-affine polynomial $p \in \K[z_1,...,z_n]$, we have that $\{z_{i_j}\}_{j=1}^m$ splits $p$ iff no term of $p$ contains more than one variable from $\{z_{i_j}\}_{j=1}^m$. This naturally leads to the use of ``variable selection'' ($z_{i_j} \partial_{z_{i_j}} p$) and ``variable deselection'' ($p_0$) in the decomposition of $p$ into the above sum of polynomials.

We now reach the main theorem of this section. As mentioned before, this can be seen as a loose analogue of the stability equivalence theorem (\ref{thm:alg_hb}) of Borcea and Br{\"a}nd{\'e}n.

% (12/21): use of nonneg coeff
\begin{theorem} \label{thm:samephase}
Let $p \in \R_+[z_1,...,z_n]$ be a multi-affine polynomial with nonnegative coefficients. The following are equivalent.
\begin{enumerate}[(i)]
    \item The polynomial $p$ is same-phase stable.
    \item Given any proper splitting
        \[
        p = f_0 + \sum_{j=1}^m z_{i_j} f_j
        \]
        we have that $f_0$, $z_{i_1}f_1$, ..., and $z_{i_m}f_m$ are same-phase compatible.
    \item There exists some proper splitting
        \[
        p = f_0 + \sum_{j=1}^m z_{i_j} f_j
        \]
        such that $f_0$, $z_{i_1}f_1$, ..., and $z_{i_m}f_m$ are same-phase compatible.
\end{enumerate}
\end{theorem}
\begin{proof}
$(i) \Rightarrow (ii)$ Let $p = f_0 + \sum_{j=1}^m z_{i_j}f_j$ be a proper splitting of $p$. By uniqueness of the proper splitting, $f_0$ is the polynomial $p$ with variables $\{z_{i_j}\}$ evaluated at 0, and $z_{i_j}f_j = z_{i_j}\partial_{z_{i_j}} p$. So, by closure properties, each of $f_0$, $z_{i_1}f_1$, ..., and $z_{i_m}f_m$ is same-phase stable. Now, fix $\mathbf{t} \in \R_+^n$ and $\boldsymbol{\lambda} \in \R_+^m$, and let $\boldsymbol{\lambda t}$ be defined as:
\[
(\boldsymbol{\lambda t})_i := \left\{
    \begin{array}{lr}
        \lambda_j t_{i_j}, & i = i_j \\
        t_i, & i \not\in \{i_j\}_{j=1}^m
    \end{array}
\right.
\]
That is, $\boldsymbol{\lambda t}$ is obtained by multiplying the $i_j$'th entry of $\mathbf{t}$ by $\lambda_j$ for all $j \in [m]$. With this, same-phase stability of $p$ implies
\[
\big(1 + \sum_j \lambda_j\big)^{-1}p(\boldsymbol{\lambda t}z) = \frac{f_0(\mathbf{t}z) + \sum_j \lambda_j [t_{i_j}z f_j(\mathbf{t}z)]}{1 + \sum_j \lambda_j}
\]
is real-rooted for every choice of $\boldsymbol{\lambda}$, which means every convex combination of $f_0(\mathbf{t}z)$, $t_{i_1}z f_1(\mathbf{t}z)$, ..., and $t_{i_m}z f_m(\mathbf{t}z)$ is real-rooted. So, $f_0(\mathbf{t}z)$, $t_{i_1}z f_1(\mathbf{t}z)$, ..., and $t_{i_m}z f_m(\mathbf{t}z)$ have a common interlacing. Since $\mathbf{t}$ was arbitrary, this implies $f_0$, $z_{i_1}f_1$, ..., and $z_{i_m}f_m$ are same-phase compatible.

$(ii) \Rightarrow (iii)$ This is trivial, given the existence of some proper splitting. In particular, $p = p(0,z_2,...,z_n) + z_1\partial_{z_1}p$ is always a proper splitting for multi-affine $p$.

$(iii) \Rightarrow (i)$ Fix $\mathbf{t} \in \R_+^n$. Same-phase compatibility of $f_0$, $z_{i_1}f_1$, ..., and $z_{i_m}f_m$ implies $f_0(\mathbf{t}z)$, $zf_1(\mathbf{t}z)$, ..., and $zf_m(\mathbf{t}z)$ have a common interlacing. So,
\[
\big(1 + \sum_j t_{i_j}\big)^{-1}p(\mathbf{t}z) = \frac{f_0(\mathbf{t}z) + \sum_j t_{i_j}z f_j(\mathbf{t}z)}{1 + \sum_j t_{i_j}}
\]
is real-rooted. Since $\mathbf{t}$ was arbitrary, this implies $p$ is same-phase stable.
\end{proof}

The power of this statement comes from the fact that same-phase compatibility of any particular splitting implies same-phase compatibility of every possible splitting. We will use this to our advantage in an inductive argument to follow.

\section{Multivariate Graph Polynomials and Stability}

In this section, we discuss the multivariate analogues of the independence and matching polynomials. Though somewhat counterintuitive, considering the multivariate versions of these polynomials actually simplifies the situation. In the multivariate world, one can directly manipulate how particular vertices and edges influence the polynomial by manipulating the associated variable. And further, these polynomials are multiaffine: important operations like differentiation and evaluation at 0 have intuitive interpretations.

Notions like real-rootedness and root bounds become trickier in the multivariate world, but real stability and similar notions can often play the analogous parts. This is true for the multivariate matching polynomial and somewhat true for the multivariate independence polynomial, as we will see below. But first, let's set up some notation.

\subsection{Notation}

Let $G=(V_G,E_G)$ be an undirected graph, which is simple unless otherwise specified. As usual, $V_G$ is the set of vertices and $E_G$ is the set of edges. We employ standard notation surrounding these first objects:
\begin{itemize}
    \item $\{u,v\} \in E_G$ iff there is an edge between vertices $u$ and $v$
    \item $u \in e$ for $e \in E_G$ iff $u$ is a vertex of the edge $e$
    \item $N_G[v]$ (resp. $N_G(v)$) denotes the closed (resp. open) neighborhood of $v$
    \item $H \subseteq G$ (resp. $H \leq G$) iff $H$ is a subgraph (resp. induced subgraph) of $G$
\end{itemize}
As usual, we will leave off the subscript $G$ when unambiguous. We also generalize the definition of ``claw'' in the following standard way. As usual, let $K_{m,n}$ denote the complete bipartite graph with $m+n$ vertices. So, we refer to $K_{1,3}$ as a \emph{claw} or as a \emph{3-star}. Generalizing, we refer to $K_{1,n}$ as an \emph{$n$-star}. For any graph $H$, we say that $G$ is \emph{$H$-free} if it does not contain $H$ as an induced subgraph.

% (1/11): path tree out (defined later) and line graph defn explicitly given
Finally, we denote the line graph of $G$ by $L(G)$. This is the graph formed by considering the edges of $G$ to be the vertices of $L(G)$, with adjacency in $L(G)$ determined by whether or not the corresponding edges of $G$ share a vertex in $G$.
%, and the path tree of $G$ with respect to $v \in V$ by $T_v(G)$. These last notions are crucial to the study of the univariate matching and independence polynomials, and we will see that their importance extends naturally to the multivariate world.

\subsection{The Matching Polynomial}

% (1/11): updated to give appropriate credit to HL and COSW
The univariate and multivariate matching polynomials have been well studied. In 1972, Heilmann and Lieb proved that for any graph the multivariate matching polynomial is real-stable. This implies the real-rootedness of the univariate matching polynomial, and in fact Heilmann and Lieb gave bounds on its largest root. More recently, Choe, Oxley, Sokal, and Wagner \cite{choe2004homogeneous} gave a simpler proof of this fact using a special linear operator on polynomials, called the ``multi-affine part''. We their proof below.

First though, we define and discuss a few multivariate matching polynomials. The reader should be aware that our notation will be slightly different from that which is standard; we do this to emphasize the connection between the matching and independence polynomials. We give examples of all these polynomials in Figure $\ref{fig:poly}$.

Given any graph $G$, we define the multi-affine \emph{vertex matching polynomial} of $G$ as follows.
\[
\mu_V(G) \equiv \mu_V(G)(\mathbf{x}) := \sum_{\substack{M \subset E \\ M, \text{matching}}} \prod_{\{u,v\} \in M} -x_ux_v
\]
Notice that the univariate restriction of $\mu_V(G)$ is the univariate matching polynomial used by Godsil and Heilmann-Lieb, but with the degrees inverted. So, for instance, Heilmann and Lieb's upper bound on the absolute value of the roots of the matching polynomial would translate to a bound \emph{away} from zero for this inverted polynomial. We will discuss this further later. We also define the multiaffine \emph{edge matching polynomial} of $G$ as follows.
\[
\mu_E(G) \equiv \mu_E(G)(\mathbf{x}) := \sum_{\substack{M \subset E \\ M, \text{matching}}} \prod_{e \in M} x_e
\]
We now give the proof of real stability of the vertex matching polynomial, and show its connection to the edge matching polynomial.

\begin{theorem}[\cite{HL}, \cite{choe2004homogeneous}, \cite{bb2}]
For any graph $G$, the vertex matching polynomial $\mu_V(G)$ is real stable.
\end{theorem}
\begin{proof}
Let $\text{MAP}$ (``Multi-Affine Part'') denote the linear operator on multivariate polynomials which removes any terms which are not multi-affine. By \cite{bb2}, this operator preserves real stability. We then have the following.
\[
\mu_V(G)(\mathbf{x}) = \text{MAP}\left(\prod_{\{u,v\} \in E} (1 - x_ux_v)\right)
\]
Since $(1-x_ux_v)$ is real stable and the product of real stable polynomials is real stable, this implies the result.
\end{proof}

This then implies real-rootedness of the univariate matching polynomial via univariate restriction. As for the edge matching polynomial, we don't quite have real stability. However, we do have same-phase stability, which still implies real-rootedness of the univariate restriction.

\begin{corollary}
For any graph $G$, the edge matching polynomial $\mu_E(G)$ is same-phase stable.
\end{corollary}
\begin{proof}
Let $\Pi^\downarrow$ be the projection operator, which sends all variables $x_v$ to a single variable $x$. Fixing $(t_e)_{e \in E} \in \R_+^{|E|}$, we have the following.
\[
\mu_E(G)(-\mathbf{t}x^2) = \sum_{\substack{M \subset E \\ M, \text{matching}}} \prod_{e \in M} -t_e x^2 = (\Pi^\downarrow \circ \text{MAP})\left(\prod_{\{u,v\} \in E} (1 - t_ex_ux_v)\right)
\]
By closure properties of real stability and the fact that $t_e > 0$ implies $(1 - t_ex_ux_v)$ is real stable, the right-hand side of the above equation is real-rooted. So, $\mu_E(G)(-\mathbf{t}x^2)$ is real-rooted, which implies $\mu_E(G)(\mathbf{t}x)$ is real-rooted. (In fact, it has all its roots on the negative part of the real line.) Since $\mathbf{t}$ was arbitrary, this implies the result.
\end{proof}

It's well-known that matchings of graphs are related to independent sets of line graphs. This connection is made particularly clear by considering the (multivariate) edge matching polynomial, as we will see in the next section.

\subsection{The Independence Polynomial}

The univariate independence polynomial of a graph is another well-studied graph polynomial. However, consideration of its roots has proven a bit more difficult. For example, the independence polynomial of a graph is \emph{not} real-rooted in general, and it has only been about a decade since the first proof of real-rootedness for claw-free graphs was published in \cite{CS}. Since then a number of proofs of real-rootedness have appeared, along with interesting results about location and modulus of certain roots (\cite{rootbounds3}, \cite{rootbounds2}, \cite{mehler}, \cite{CDidentities}).

Here, we give another proof of real-rootedness for claw-free graphs by proving something stronger: namely, that the multivariate independence polynomial of a graph is same-phase stable if and only if the graph is claw-free. In their original proof, Chudnovsky and Seymour show real-rootedness using an intricate recursion based on a combinatorial structure known as a ``simplicial clique''. By encoding the recursive compatibility using our notion of same-phase stability, we are able to avoid the introduction of simplicial cliques and use simpler graph structures in the recursion. Same-phase stability of the edge matching polynomial then serves as the base case.

Before giving this proof, we need to set up the relevant notation. Given any graph $G$, we define the multi-affine \emph{independence polynomial} of $G$ as follows.
\[
I(G) \equiv I(G)(\mathbf{x}) := \sum_{\substack{S \subset V \\ S, \text{independent}}} \prod_{v \in S} x_v
\]

Stability properties of the multivariate independence polynomial have been previously studied by Scott and Sokal. In \cite{lll}, they observe this polynomial as a specific case of a more general statistical-mechanical partition function, and generic lower bounds on the modulus of the roots are studied. In particular, the Lov{\'a}sz local lemma is used to give a universal lower bound of $\frac{1}{e \cdot \Delta}$, where $\Delta$ is the maximum degree of $G$.

\begin{figure}
    \centering
    \begin{tabular}{m{1.5in}  m{4in}}
       
         \begin{tikzpicture}
create the node
\node[draw=black,minimum size=3cm,regular polygon,regular polygon sides=6] (p) {};

\node[label={a}, draw=black,fill,circle,inner sep=0pt,minimum size=3pt] (a) at (p.corner 1) {};

\node[label={b}, draw=black,fill,circle,inner sep=0pt,minimum size=3pt] (b) at (p.corner 2) {};

\node[label=left:{c}, draw=black,fill,circle,inner sep=0pt,minimum size=3pt] (c) at (p.corner 3) {};

\node[label=below:{d}, draw=black,fill,circle,inner sep=0pt,minimum size=3pt] (d) at (p.corner 4) {};

\node[label=below:{e}, draw=black,fill,circle,inner sep=0pt,minimum size=3pt] (e) at (p.corner 5) {};

\node[label=right:{f}, draw=black,fill,circle,inner sep=0pt,minimum size=3pt] (f) at (p.corner 6) {};
\end{tikzpicture}
& \begin{tabular}{l} 
\parbox{3.7in}{$\mu_E(C_6,x) = 1 + x_{ab} + x_{bc} + x_{cd} + x_{de} + x_{ef} + x_{fa} + x_{ab} x_{cd} + x_{ab} x_{de} + x_{ab} x_{ef} + x_{bc} x_{de} + x_{bc} x_{ef} + x_{bc} x_{fa} + x_{cd} x_{ef} + x_{cd} x_{fa} + x_{de} x_{fa} + x_{ab} x_{cd} x_{ef} + x_{bc} x_{de} x_{fa}$}\\ \\
\parbox{3.7in}{$\mu_V(C_6,x) = 1 - x_{a}x_{b} - x_{b}x_{c} - x_{c}x_{d} - x_{d}x_{e} - x_{e}x_{f} - x_{f}x_{a} + x_{a}x_{b} x_{c}x_{d} + x_{a}x_{b} x_{d}x_{e} + x_{a}x_{b} x_{e}x_{f} + x_{b}x_{c} x_{d}x_{e} + x_{b}x_{c} x_{e}x_{f} + x_{b}x_{c} x_{f}x_{a} + x_{c}x_{d} x_{e}x_{f} + x_{c}x_{d} x_{f}x_{a} + x_{d}x_{e} x_{f}x_{a} - x_{a}x_{b} x_{c}x_{d} x_{e}x_{f} - x_{b}x_{c} x_{d}x_{e} x_{f}x_{a} $}\\ \\
\parbox{3.7in}{$I(C_6,x) = 1 + x_a + x_b + x_c + x_d + x_e + x_f + x_a x_c + x_a x_d + x_a x_e + x_b x_d + x_b x_e + x_b x_f + x_c x_e + x_c x_f + x_d x_f + x_a x_c x_e + x_b x_d x_f$}\\
\end{tabular}
    \end{tabular}
    
    \caption{A small graph $C_6$ with associated independence polynomial, vertex/edge matching polynomials.}
    \label{fig:poly}
\end{figure}
As discussed in the notation above, for a given graph $G$ we denote the line graph of $G$ by $L(G)$. Since line graphs are claw-free, we have the following first step toward the desired result.

\begin{corollary}
For any graph $G$, the independence polynomial $I(L(G))$ of the line graph of $G$ is same-phase stable.
\end{corollary}
\begin{proof}
By considering the fact that the operator $L$ maps edges to vertices and shared vertices to edges, we actually have the following identity.
\[
\mu_E(G) = I(L(G))
\]
The previous corollary gives the desired result.
\end{proof}

Of course, this is quite far from the claim that \emph{all} claw-free graphs are same-phase stable. However, as it turns out, line graphs will serve a base case in our induction on general claw-free graphs. To illustrate this, we first give the following lemma.

\begin{lemma}
Let $G$ be a connected claw-free graph which is also triangle-free. Then, $G$ is either a path or a cycle. In particular, $G$ is a line graph.
\end{lemma}
\begin{proof}
Given a vertex $v \in G$, if the degree of $v$ is greater than 2 then we get either a claw with $v$ as the base or a triangle. We conclude that a graph which is connected, claw-free, and triangle-free is equivalent to being connected and triangle-free with all vertices degree 1 or 2.

% Since the sum of the degrees of the vertices is $2 |E|$, and a connected graph satisfies $|E| \geq |V| - 1$, we get $2|V| \geq \sum_v \deg(v) \geq 2 |V| - 2$. This implies we have either 0 or 2 vertices of degree one, and the rest are of degree $2$. 

% If we have $2$ vertices of degree one, starting at either one we get a unique path leading out from it, which must terminate at the other for the graph to be connected. Hence in this case our graph is a path. If we have $0$ vertices of degree one, then we can start at any vertex and extend a path in one direction. This must eventually return to the starting vertex to avoid creating vertices of degree one or more than two. Hence in this case we have a cycle.

% The path with $n$ vertices is the line graph of the path with $n+1$ vertices, while the cycle with $n$ vertices is the line graph of itself.
\end{proof}

% (12/21): reference to Engstrom, and in Theorem below
With this, we now give the proof of same-phase stability for claw-free graphs, using the theory of same-phase compatibility developed above. As mentioned in the introduction, this result is a reformulation of a theorem of Engstr\"om given in \cite{weighted}. 

% (12/21): cleaned up the last few sentences of the proof
\begin{theorem}[Engstr\"om]
For any claw-free graph $G$, the independence polynomial $I(G)$ is same-phase stable.
\end{theorem}
\begin{proof}
We induct on the number of vertices. If $G$ is disconnected, then its independence polynomial is the product of the independence polynomials of its connected components. The inductive hypothesis on components of $G$ (along with the disjoint product closure property for same-phase stable polynomials) then implies the result for $G$. If $G$ is connected and contains no 3-cliques (triangles), then $G$ is a line graph by the previous lemma. The line graph corollary then implies the result for $G$. If neither of these conditions is satisfied, then $G$ is a connected graph with at least one 3-clique. Let $u,v,w$ denote the vertices of this 3-clique.

In the independence polynomial $I(G)$, let the variables $z_u,z_v,z_w$ represent the vertices $u,v,w$, respectively. Consider the following equivalent expressions of $I(G)$.
\[
\begin{split}
    I(G)
    &= \left.I(G)\right|_{u=v=w=0} + z_u \partial_{z_u}I(G) + z_v \partial_{z_v}I(G) + z_w \partial_{z_w}I(G) \\
    &= I(G \setminus \{u,v,w\}) + z_u I(G \setminus N[u]) + z_v I(G \setminus N[v]) + z_w I(G \setminus N[w]) \\
    &= \left[I((G \setminus \{u\}) \setminus \{v,w\}) + z_v I((G \setminus \{u\}) \setminus N[v]) + z_w I((G \setminus \{u\}) \setminus N[w])\right] + z_u I(G \setminus N[u]) \\
    &= \left[I((G \setminus \{v\}) \setminus \{u,w\}) + z_u I((G \setminus \{v\}) \setminus N[u]) + z_w I((G \setminus \{v\}) \setminus N[w])\right] + z_v I(G \setminus N[v]) \\
    &= \left[I((G \setminus \{w\}) \setminus \{u,v\}) + z_u I((G \setminus \{w\}) \setminus N[u]) + z_v I((G \setminus \{w\}) \setminus N[v])\right] + z_w I(G \setminus N[w]) \\
\end{split}
\]
The square-bracketed sections of the last three expressions are proper splittings of $I(G \setminus \{u\})$, $I(G \setminus \{v\})$, and $I(G \setminus \{w\})$, respectively. By the inductive hypothesis and the same-phase stability theorem, these proper splittings have terms which are same-phase compatible. So, the terms of the first expression of $I(G)$ are pairwise same-phase compatible. By Corollary \ref{cor:pw_samephase_compat}, we have that all the terms of the first expression are same-phase compatible. These terms give a proper splitting of $I(G)$, and so Theorem \ref{thm:samephase} implies $I(G)$ is same-phase stable.
\end{proof}

An interesting feature of the above proof is the fact that the inductive step did not use the fact that $G$ is claw-free. This suggests that perhaps the theorem can be extended to certain clawed graphs. However, the following corollary shows that this is not the case.

\begin{corollary}
For any graph $G$, the independence polynomial $I(G)$ is same-phase stable if and only if $G$ is claw-free (3-star-free).
\end{corollary}
\begin{proof}
By the above theorem, we only need to show that the independence polynomial of a graph with a claw is not same-phase stable. To get a contradiction, let $G$ be a graph such that the vertices $u,v,w,x$ form a claw, and yet $I(G)$ is same-phase stable. Let $p(z_u,z_v,z_w,z_x)$ be the polynomial obtained by evaluating $I(G)$ at zero for all other variables (besides $z_u$, $z_v$, $z_w$, and $z_x$). By closure properties, $p$ is also same-phase stable. With this we compute $p(\mathbf{t}z)$ for $\mathbf{t} = (1,1,1,1)$:
\[
p(z,z,z,z) = 1 + 4z + 3z^2 + z^3
\]
This polynomial is not real-rooted, which gives the desired contradiction.
\end{proof}

With this equivalence in mind, one might wonder for what smaller class of graphs the independence polynomial is actually real stable. A somewhat surprising result is the following.

\begin{proposition}
For any connected graph $G$, the independence polynomial $I(G)$ is real stable if and only if $G$ is complete (2-star-free).
\end{proposition}
\begin{proof}
If $G$ is a complete graph, then the independence polynomial of $G$ is $1 + \sum_{v \in V} x_v$, which is real stable. On the other hand, suppose $G$ is some connected incomplete graph such that $I(G)$ is real stable. By incompleteness and connectedness, $G$ contains an induced path $P$ of length at least 2. (E.g., consider the shortest path between two non-adjacent vertices.) In fact, we can assume $P$ is of length exactly 2 by removing all but 3 consecutive vertices. Notice that $P$ is now an induced 2-star. Evaluating $I(G)$ at 0 the variables $x_v$ for which $v \not\in P$, we obtain $I(P)$, the independence polynomial of $P$. Closure properties imply $I(P)$ is real stable.

Labeling the vertices of $P$ as $u,v,w$, we then have
\[
I(P)(\mathbf{x}) = 1 + x_u + x_v + x_w + x_ux_w,
\]
which, for $\mathbf{x}_0 = (-1,1,-1)$, gives
\[
\partial_{x_u}I(P)(\mathbf{x}_0) \cdot \partial_{x_w}I(P)(\mathbf{x}_0) = 0 < 1 = \partial_{x_ux_w}I(P)(\mathbf{x}_0) \cdot I(P)(\mathbf{x}_0).
\]
That is, $I(P)$ is not strongly Rayleigh. So, $I(P)$ is not real stable, which is the desired contradiction.
\end{proof}

\section{Root Bounds}

In addition to proving real rootedness of the matching polynomial, Heilmann and Lieb established bounds on the modulus of roots of the matching polynomial. Since we use the inverted matching polynomial, this result bounds the roots, $\lambda$, of $\mu_V(G)$ away from zero:
$$|\lambda| \geq \frac{1}{2 \sqrt{\Delta - 1}}$$

Since $\mu_V(G)(x) = I(L(G))(-x^2)$ this result can be stated equivalently as a bound on the root closest to zero, $\lambda_1$, for the independence polynomial of line graphs. To do this note that the maximum degree, $\Delta$, of a graph is equal to the clique size, $\omega$, of its line graph.
$$\lambda_1(I(L(G))) \leq \frac{1}{4(\omega-1)}$$

Since all line graphs are claw-free graphs, we can seek out similar bounds for the independence polynomial of claw-free graphs. In what follows, we adapt the methods of Godsil to determine such root bounds for a certain subclass of claw-free graphs, namely those which contain a simplicial clique. (Although we were able to avoid simplicial cliques in the proof of real-rootedness, they turn out to be crucial to generalizing the Heilmann-Lieb root bound.) We then discuss how the bound does not extend to all claw-free graphs.

To this end, we first discuss Godsil's original divisibility result which was key to his proof of the Hielmann-Lieb root bound. We do this in the multivariate world, though, so as to provide context for the later results on the independence polynomial.

\subsection{Path Trees}

A basic element of Godsil's proof of the root bound is the notion of a path tree of a graph. We now define this notion as he did, and subsequently discuss what needs to be altered in order to apply it to the multivariate matching polynomial.

\begin{definition}
Given a graph $G$ and vertex $v$, we define the \emph{(labeled) path tree} $T_v(G)$ of $G$ with respect to $v$ recursively as follows. If $G$ is a tree, we define $T_v(G) = G$, and we say that $v$ is the root of $T_v(G)$. We also label the vertices of $T_v(G)$ using the vertices of $G$. (In the recursive step, we will continue to label using vertices of $G$.)

For an arbitrary graph $G$, we first consider the forest which is the disjoint union of the labeled trees $T_w(G \setminus \{v\})$ for each $w \in N(v)$. We then define $T_v(G)$ by appending a vertex (the root) labeled $v$ and connecting it to the roots of each of these trees. 
\end{definition}
\begin{remark}
Figure $\ref{fig:example1}$ gives an example of a path tree. Note that it is defined in such a way that the paths stemming from $v$ in $G$ and from the root, $v$ in $T_v(G)$, are in order preserving bijection (where the order on paths is the subpath ordering).
\end{remark}
% Anymore to say?

% Given a graph $G$ and vertex $v$, we define the path tree $T_v(G)$ of $G$ with respect to $v$ as Godsil did in \cite{godsil}. Additionally, we give a ``labeling'' of the vertices of $T_v(G)$, the labels given by vertices in $G$. We define this ``labeled path tree'' recursively as follows. First, we add a vertex to $T_v(G)$ (corresponding to $v$) which we label with $v$. Next, for each neighbor $u \in G$ of $v$, we consider the graph $G \setminus \{v\}$ (to be used in the next recursion step), and we add a vertex to $T_v(G)$ which we label with $u$. We connect with an edge this $u$-labeled vertex to the $v$-labeled vertex we added previously. We now recurse. For each of the $u$-labeled vertices we added, we consider each neighbor $w \in G \setminus \{v\}$ of $u$. For each such $w$, we consider the graph $G \setminus \{v,u\}$, and we add a vertex to $T_v(G)$ which we label with $w$. We connect with an edge this $w$-labeled vertex to the $u$-labeled vertex we added previously. We continue this process until the vertices of $G$ have been removed in every path of the recursion. 

% We note a few things. First of all, this process produces a tree that describes all possible acyclic paths in $G$ starting at $v$. Second, with respect to this interpretation of $T_v(G)$, the labeling is intuitive. In particular, you label the vertices of the path tree by the vertices of $G$ traversed by these acyclic paths.

In Godsil's proof of the root bound for the matching polynomial, he shows that the univariate vertex matching polynomial of $G$ divides that of $T_v(G)$ for any $v$. In the multivariate world, this divisibility relation won't be possible, a priori, since there are potentially far more vertices (and hence, variables) in $T_v(G)$ than in $G$. However, using the labeling of the vertices described above, we can in fact extend this divisibility result. We now formalize this notion of labeling, so as to easily generalize it to all relevant multivariate graph polynomials.

Let $G,H$ be two graphs, and let $\phi: G \rightarrow H$ be a graph homomorphism. We call this homomorphism a \emph{labeling of $G$ by $H$}. For a graph $G$, we define the \emph{relative vertex matching polynomial} (with respect to $\phi$) as follows.
\[
\mu^\phi_V(G) \equiv \mu^\phi_V(G)(\mathbf{x}) := \sum_{\substack{M \subset E(G) \\ M, \text{matching}}} \prod_{\{u,v\} \in M} -x_{\phi(u)}x_{\phi(v)}
\]
We define the \emph{relative edge matching polynomial} and the \emph{relative independence polynomial} (with respect to $\phi$) analogously. When unambiguous, we will remove the $\phi$ superscript from the notation. Notice that the univariate specialization of each of the normal matching and independence polynomials is the same as that of the relative matching and independence polynomials, for any $\phi$. This notion then gives us a way to compare multivariate matching and independence polynomials from different graphs without destroying any univariate information.

% (12/21): cleaned up this paragraph
Now, consider the labeling of vertices described in the construction of $T_v(G)$ above. This can extended to a graph homomorphism, $\phi_v: T_v(G) \to G$ in a unique way. Specifically, the vertices of $T_v(G)$ are mapped to the vertices of $G$ via the labeling given above (e.g., the root of $T_v(G)$ maps to $v \in G$, the neighbors of the root are mapped to the neighbors of $v \in G$, etc.). An edge $\{u,w\}$ of $T_v(G)$ is then mapped to the edge $\{T_v(u), T_v(w)\}$ in $G$, which exists by the inductive construction given above.
%by looking at the edges between adjacent vertices in $G$. Letting $\phi_v: T_v(G) \rightarrow G$ denote this homomorphism,

In what follows, we will consider the graph polynomials $\mu^{\phi_v}_V(T_v(G))$ and $\mu^{\phi_v}_E(T_v(G))$. For simplicity of notation, we will from now on denote these polynomials $\mu_V(T_v(G))$ and $\mu_E(T_v(G))$, respectively. That is, reference to $\phi_v$ will be dropped.

With this, we now state the generalization of Godsil's divisibility theorem for the vertex matching polynomial. We omit the proof, as this theorem turns out to be a corollary of a more general result related to independence polynomials.

\begin{theorem}[Godsil]
Let $v$ be a vertex of the graph $G = (V,E)$, and let $T \equiv T_v(G)$ be the path tree of $G$ with respect to $v$. Further, let $\mu_V(T) \equiv \mu^{\phi_v}_V(T)$ denote the relative vertex matching polynomial. We then have the following.
\[
\frac{\mu_V(G)}{\mu_V(G \setminus v)} = \frac{\mu_V(T)}{\mu_V(T \setminus v)}
\]
Further, $\mu_V(G)$ divides $\mu_V(T)$.
\end{theorem}

By univariate specialization, this gives us the first step toward the well-known Heilmann and Lieb root bound (up to inversion of the input variable). We now attempt to generalize this divisibility to independence polynomials. First, however, we will need to develop some path tree analogues.

\subsection{Path Tree Analogues}

\subsubsection*{Induced Path Trees}
%%Should we make this a definition as well?

Given a graph $G$ and a vertex $v$, the \emph{induced path tree} $T^\angle_v(G)$ of $G$ with respect to $v$ is intuitively defined as follows: it is the path tree that is constructed when only \emph{induced} paths are considered. That is, we use the recursive process of creating the usual path tree, only we forbid traversal of vertices which are neighbors of previously traversed vertices. So, another name that could be used for this tree is the ``neighbor-avoiding'' path tree.

We now give an explicit definition of the induced path tree. The crucial difference between this definition and the definition of the path tree given above is that neighbors of a vertex are excluded in the recursive step.

\begin{definition}
Given a graph $G$ and vertex $v$, we define the \emph{induced path tree} $T^\angle_v(G)$ of $G$ with respect to $v$ recursively as follows. If $G$ is a tree, we define $T^\angle_v(G) = G$, and we say that $v$ is the root of $T^\angle_v(G)$.

For an arbitrary graph $G$, we first consider the forest which is the disjoint union of the trees $T^\angle_w(G \setminus N[v] \cup \{w\})$ for each $w \in N(v)$. We then define $T^\angle_v(G)$ by appending a vertex corresponding to $v$ (the root) and connecting it to the roots of each of these trees.
\end{definition}

% Explicitly, we can recurse as we did with the path tree. For trees $G$ we define $T^\angle_v(G) = T_v(G)$. For an arbitrary graph $G$ we form the subtrees of our root $v$ from the smaller rooted trees $T_w(G \setminus N[v] \cup \{w\})$ for every $w \in N(v)$. We also label the vertices of the induced path tree in the same way as for the usual path tree. As a note, $T^\angle_v(G)$ will be a subgraph of $T_v(G)$ in general, and if $G$ is a tree, then the path tree and the induced path tree will coincide.

We also define a slightly different version of the induced path tree. As will be seen, this adjusted definition is more appropriate for our purposes.

\begin{definition}
Given a graph $G$ and a clique $K$, the \emph{induced path tree} $T^\angle_K(G)$ of $G$ with respect to $K$ is defined as follows. Construct a new graph $G^*$ by attaching a new vertex $*$ to $G$, with the property that $\{*,u\} \in E(G^*)$ iff $u \in K$. Then, define $T^\angle_K(G) := T^\angle_{\{*\}}(G^*)$.
\end{definition}

\begin{remark}
As with the path tree, we can label the vertices of the induced path tree in a natural way. This gives rise to graph homomorphisms $\phi_v: T^\angle_v(G) \rightarrow G$ and $\phi_K: T^\angle_K(G) \rightarrow G^*$.
\end{remark}

\subsubsection*{Simplicial Clique Trees}
%%Maybe we should call it simplicial clique tree, since block graph = clique tree in the literature
%%Should we use setminus or backslash?
%%Should we mention what is special about a simplicial block graph?
%%In the diagram should we make it rooted trees?
%% Is my remark neccessary/does it make sense?

% (12/21): added parentheses (and below similarly)
We need two graph theoretic concepts before defining our final path tree analogue. Given a graph $G$, let $K \leq G$ be an induced clique. Then, $K$ is called a \emph{simplicial clique} if for all $u \in K$, $N[u] \cap (G \setminus K)$ is a clique as an induced subgraph of $G$ (or equivalently, as an induced subgraph of $G \setminus K$). Intuitively, this means that neighborhoods of each $u \in K$ are two cliques joined at $u$: one is $K$ itself, and the other consists of the remaining neighbors of $u$. Simplicial cliques have been studied frequently in relation to the independence polynomial of a graph, and in particular, they were used in Chudnovsky and Seymour's original proof of real-rootedness for claw-free graphs.

We further say that a graph $G$ is \emph{simplicial} if it is claw-free and contains a simplicial clique. It may at first seem strange as to why ``claw-free'' is included in this definition. The main reason is the useful recursive structure that can be extracted from the following lemma.

\begin{lemma}[\cite{CS}]
Let $G$ be claw-free, and let $K \leq G$ be a simplicial clique in $G$. For any $u \in K$, $N[u] \cap (G \setminus K)$ is a simplicial clique in $G \setminus K$.
\end{lemma}
\begin{remark}
One can easily check that our definition of a simplicial graph is equivalent to having a recursive structure of simplicial cliques as indicated in the previous lemma.
\end{remark}

A \emph{block graph} (or \emph{clique tree}) is a graph in which every maximal 2-connected subgraph is a clique \cite{block}. As it turns out, block graphs are precisely the line graphs of trees. From this observation we note that there is a natural tree-like recursive structure on block graphs. Specifically, let $B$ be a block graph, and let $K$ be a clique in $B$. Then, $B \setminus K$ is a ``forest of block graphs''. That is, if we refer to $K$ as the ``root clique'' in $B$, then each ``root clique'' in the forest $B \setminus K$ is connected to some vertex of $K$ in $B$.

We now define a special kind of clique tree. Notice that while the term ``tree'' is used, the graphs defined here are not actually trees in the usual sense.

\begin{definition}
Given a simplicial graph $G$ and simplicial clique $K \leq G$, we define the \emph{(simplicial) clique tree} $T^\boxtimes_K(G)$ of $G$ with respect to $K$ recursively as follows. If $G = K$, we define $T^\boxtimes_K(G) = G$, and we say that $K$ is the ``root clique'' of $T^\boxtimes_K(G)$.

For an arbitrary graph $G$, we first consider the ``forest of simplicial clique trees'' which is the disjoint union of $T^\boxtimes_{J_u}(G \setminus K)$ for each $u \in K$. (Here, we define $J_u := N[u] \cap (G \setminus K)$.) Note that this is valid, since the previous lemma implies $J_u$ is a simplicial clique for all $u \in K$. We then define $T^\boxtimes_K(G)$ by appending the clique $K$ (the root clique) and connecting each vertex $u \in K$ to each vertex of the root clique of $T^\boxtimes_{J_u}(G \setminus K)$.
\end{definition}
\begin{remark}
We can label the vertices of the (simplicial) clique tree in the usual way, and this gives rise to a natural graph homomorphism $\phi_K: T^\boxtimes_K(G) \rightarrow G$.
\end{remark}

For examples of the induced path tree and the simplicial clique tree, see Figures $\ref{fig:example1}$ and $\ref{fig:example2}$.

\subsection{Divisibility Relations}

Given the above definitions, the main goal of this section is to demonstrate the following theorem. Here, for $v \in G$ we define $K_v \leq L(G)$ via $K_v := L(\{e \in E(G): v \in e\})$. That is, $K_v$ can be thought of as ``the clique in $L(G)$ associated to $N[v]$''.

\begin{theorem} \label{thm:diagram}
Let $L$ be the line graph operator, $T_v$ the path tree operator with respect to $v$, $T^\angle_K$ the induced path tree operator with respect to $K$, and $T^\boxtimes_K$ the clique tree operator with respect to $K$. Then the following diagram commutes up to isomorphism.
\[
\begin{tikzcd}
    \{\text{graphs}\}
        \arrow{r}{T_v}
        \arrow{d}{L} &
    \{\text{trees}\}
        \arrow{d}{L} \\
    \{\text{simplicial graphs}\}
        \arrow{ru}{T^\angle_K}
        \arrow{r}{T^\boxtimes_K} &
    \{\text{simpl. block graphs}\}
\end{tikzcd}
\]
In the upper left triangle, commutativity is achieved for $K = K_v$.
\end{theorem}

This can be broken down into a few results, which we give now.

\begin{lemma}
For any graph $G$ and any $v \in G$, $K_v$ is a simplicial clique of $L(G)$. In particular, $L(G)$ is simplicial.
\end{lemma}
\begin{proof}
It is easy to see that $K_v$ is a clique. If we consider $w \in K_v$, this corresponds to an edge $e_w \in E(G)$ that has $v$ as an endpoint. Then given any two neighbors of $w$ that are not in $K_v$, we know they correspond to two edges which share an endpoint with $e_w$ but do not have $v$ as an endpoint. Hence they both share the other endpoint of $e_w$ and are therefore connected in the line graph. This shows that $N[w] \setminus K_v$ is a clique, so $K_v$ is a simplicial clique.

It is well known that line graphs are claw-free, so all line graphs are simplicial.
\end{proof}

\begin{proposition}
For any (nonempty) graph $G$ and any $v \in V$, the induced path tree of $L(G)$ with respect to $K_v$ is isomorphic to the path tree of $G$ with respect to $v$. That is, $T^\angle_{K_v}(L(G)) \cong T_v(G)$.
\end{proposition}
\begin{proof}
First, let $G$ be the graph with one vertex, $v$. Then, $L(G)$ is the empty graph and $T^\angle_{K_v} \circ L(G)$ is also the graph with one vertex (recall that the operator $T^\angle_{K_v}$ adds an extra vertex to the input graph). On the other hand, $T_v(G)$ is the graph with one vertex, and the result holds in this case.

Now, let $G$ be a connected graph consisting of two or more vertices, and let $v$ be some vertex of $G$. (We can assume WLOG that $G$ is connected, since $T_v$ and $T^\angle_{K_v}$ only deal with connected components of $v$ and $K_v$, respectively.) We proceed inductively, adopting the convention that $K_u \leq L(G)$ and $K'_u \leq L(G \setminus \{v\})$ are the cliques associated to $N[u]$ in the respective line graphs.

We first consider $T_v(G)$. For each $u \in N(v)$, we have that $T_u(G \setminus \{v\})$ is naturally a subtree of $T_v(G)$. In fact, $T_v(G)$ can be viewed as the disjoint union of $T_u(G \setminus \{v\})$ for all $u \in N(v)$, connected to a single vertex corresponding to $v$.

We next consider $T^\angle_{K_v} \circ L(G)$. Notice that $L(G \setminus \{v\}) \cong L(G) \setminus K_v$. For any $u \in N(v)$, this implies $T^\angle_{K'_u} \circ L(G \setminus \{v\}) \cong T^\angle_{J_u}(L(G) \setminus K_v)$, where $J_u := K_u \cap (L(G) \setminus K_v)$. Recall that the $T^\angle_K$ operator adds an extra vertex attached to each vertex of $K$. So, we can view $T^\angle_{K_v} \circ L(G)$ as the disjoint union of $T^\angle_{J_u}(L(G) \setminus K_v)$ for all $u \in N(v)$, along with an extra vertex connected to each of the added extra vertices in the disjoint union.

By the induction hypothesis, we have $T_u(G \setminus \{v\}) \cong T^\angle_{K'_u} \circ L(G \setminus \{v\})$ for all $u \in N(v)$. This implies that the two descriptions given above of $T_v(G)$ and $T^\angle_{K_v} \circ L(G)$, respectively, are equivalent. Therefore, $T_v(G) \cong T^\angle_{K_v} \circ L(G)$.
\end{proof}

\begin{proposition}
For any simplicial graph $G$ and any simplicial clique $K \leq G$, the line graph of the induced path tree of $G$ with respect to $K$ is isomorphic to the clique tree of $G$ with respect to $K$. That is, $L(T^\angle_K(G)) \cong T^\boxtimes_K(G)$.
\end{proposition}
\begin{proof}
There is a natural grading on the edges of $T^\angle_K(G)$, where the edges from $*$ to vertices in $K$ have grading $1$, and edges from vertices $v \in K$ to vertices in $N[v] \setminus K$ have grading $2$, and so forth. Then under the line graph operation we get a grading on the vertices of $L \circ T^\angle_K(G)$. 

Similarly $T^\boxtimes_K(G)$ has a natural grading on the vertices by grading $K$ as grade $1$, and for every vertex $v \in K$, grading the clique $N[v] \setminus K$ as grade $2$, and so forth. 

Now we can induct on the number of vertices in $G$. The result is obviously true for the graph with one vertex. It is then clear that the first grades of $L \circ T^\angle_K(G)$ and $T^\boxtimes_K(G)$ are isomorphic: they are both cliques of size $K$. We then label the vertices of the first grade in $L \circ T^\angle_K(G)$ by vertices in $K$ as follows. Each vertex of the first grade comes from an edge in $T^\angle_K(G)$ of the form $\{*, v\}$, for some $v \in K$. So, we label this first-grade vertex in $L \circ T^\angle_K(G)$ by ``$v$''.

In $L \circ T^\angle_K(G)$, this vertex labeled ``$v$'' connects to edges in $G$ from $v$ to vertices in $N[v] \setminus K$ in $T^\angle_K(G)$. In this way we see viewing the sub-clique tree (obtained by looking at $v$ and all of the grades below it) rooted at the vertex labeled $v$ in $L \circ T^\angle_K(G)$  is $L \circ T^\angle_{N[v] \setminus K}(G)$. Likewise by looking at the vertex labeled $v$ in $T^\boxtimes_K(G)$ we see the sub-clique tree obtained by looking at $v$ and all grades below it is exactly $T^\boxtimes_{N[v]\setminus K}(G)$, by definition of the simplicial clique tree. By induction our claim is proved.
\end{proof}

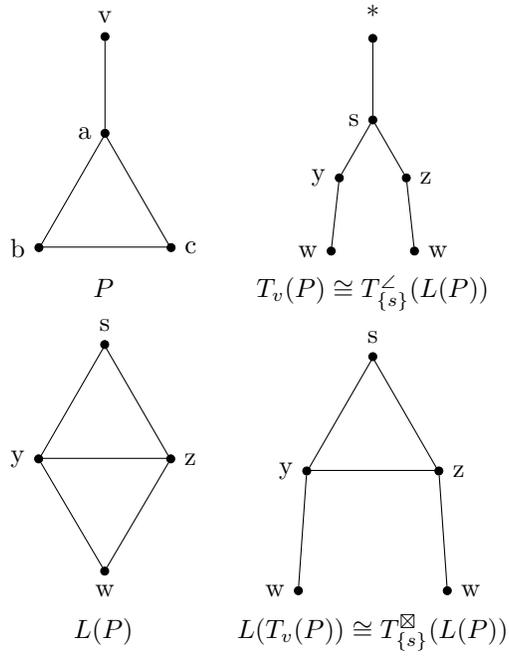
\begin{figure}
    \centering
    \begin{tabular}{ cc }
    % \hline 
\begin{tikzpicture}
% create the node
\node[draw=black,minimum size=2cm,regular polygon,regular polygon sides=3] (p) {};

\node[label=left:{a}, draw=black,fill,circle,inner sep=0pt,minimum size=3pt] (a) at (p.corner 1) {};

\node[label=left:{b}, draw=black,fill,circle,inner sep=0pt,minimum size=3pt] (b) at (p.corner 2) {};

\node[label=right:{c}, draw=black,fill,circle,inner sep=0pt,minimum size=3pt] (c) at (p.corner 3) {};

\node[label={v}, draw=black,fill,circle,inner sep=0pt,minimum size=3pt] (v) at (0,2.3) {};
% draw a black dot in each vertex
% \foreach \x in {1,2,...,5}
%   \fill (a.corner \x) circle[radius=4pt];

\draw (v) -- (a);
\end{tikzpicture}
  & 

\begin{tikzpicture}[scale=1.23]
% create the node
\node[draw=none,minimum size=1cm,regular polygon,regular polygon sides=3] (p) {};

\node[label={*}, draw=black,fill,circle,inner sep=0pt,minimum size=3pt] (*) at (0,1.3) {};

\node[label=left:{s}, draw=black,fill,circle,inner sep=0pt,minimum size=3pt] (x) at (p.corner 1) {};

\node[label=left:{y}, draw=black,fill,circle,inner sep=0pt,minimum size=3pt] (y) at (p.corner 2) {};

\node[label=right:{z}, draw=black,fill,circle,inner sep=0pt,minimum size=3pt] (z) at (p.corner 3) {};

\node[label=right:{w}, draw=black,fill,circle,inner sep=0pt,minimum size=3pt] (w1) at (0.45,-1) {};

\node[label=left:{w}, draw=black,fill,circle,inner sep=0pt,minimum size=3pt] (w2) at (-0.45,-1) {};

\draw (*) -- (x);
\draw (x) -- (y);
\draw (x) -- (z);
\draw (y) -- (w2);
\draw (z) -- (w1);

% draw a black dot in each vertex
% \foreach \x in {1,2,...,5}
%   \fill (a.corner \x) circle[radius=4pt];

\end{tikzpicture}
\\ %\hline

$P$ & $T_v(P) \cong T^\angle_{\{s\}}(L(P))$ \\

\begin{tikzpicture}
% create the node
\node[draw=black,minimum size=2cm,regular polygon,regular polygon sides=3] (p) {};

\node[label={s}, draw=black,fill,circle,inner sep=0pt,minimum size=3pt] (x) at (p.corner 1) {};

\node[label=left:{y}, draw=black,fill,circle,inner sep=0pt,minimum size=3pt] (y) at (p.corner 2) {};

\node[label=right:{z}, draw=black,fill,circle,inner sep=0pt,minimum size=3pt] (z) at (p.corner 3) {};

\node[label=below:{w}, draw=black,fill,circle,inner sep=0pt,minimum size=3pt] (w) at (0,-2) {};

\draw (y) -- (w);
\draw (z) -- (w);
\end{tikzpicture}

&

\begin{tikzpicture}[scale=1.5]
% create the node
\node[draw=black,minimum size=2cm,regular polygon,regular polygon sides=3] (p) {};

\node[label={s}, draw=black,fill,circle,inner sep=0pt,minimum size=3pt] (x) at (p.corner 1) {};

\node[label=left:{y}, draw=black,fill,circle,inner sep=0pt,minimum size=3pt] (y) at (p.corner 2) {};

\node[label=right:{z}, draw=black,fill,circle,inner sep=0pt,minimum size=3pt] (z) at (p.corner 3) {};

\node[label=right:{w}, draw=black,fill,circle,inner sep=0pt,minimum size=3pt] (w1) at (0.66,-1.4) {};

\node[label=left:{w}, draw=black,fill,circle,inner sep=0pt,minimum size=3pt] (w2) at (-0.66,-1.4) {};

\draw (y) -- (w2);
\draw (z) -- (w1);
\end{tikzpicture} \\

$L(P)$ & $L(T_v(P)) \cong T^\boxtimes_{\{s\}}(L(P))$ \\ %\hline

    \end{tabular}
    \caption{An example of a graph and its line graph, induced path tree and simplicial clique tree as in Theorem \ref{thm:diagram}.}
    \label{fig:example1}
\end{figure}

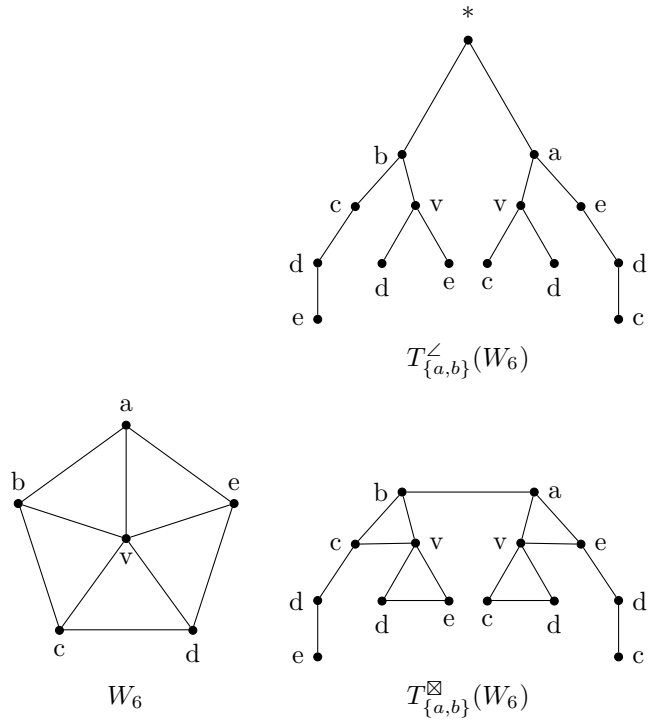
\begin{figure}
    \centering
    \begin{tabular}{ cc }
      & 
\begin{tikzpicture}
% create the node
\node[draw=none,minimum size=2cm,regular polygon,regular polygon sides=3] (t1) {};
\node[label={*}, draw=black,fill,circle,inner sep=0pt,minimum size=3pt] (s1) at (t1.corner 1) {};
\node[label=left:{b}, draw=black,fill,circle,inner sep=0pt,minimum size=3pt] (b1) at (t1.corner 2) {};
\node[label=right:{a}, draw=black,fill,circle,inner sep=0pt,minimum size=3pt] (a1) at (t1.corner 3) {};

\node[draw=none,minimum size=1cm,regular polygon,regular polygon sides=3] (t2) at (-0.7,-1.7) {};
\node[label=right:{v}, draw=black,fill,circle,inner sep=0pt,minimum size=3pt] (v2) at (t2.corner 1) {};
\node[label=below:{d}, draw=black,fill,circle,inner sep=0pt,minimum size=3pt] (d2) at (t2.corner 2) {};
\node[label=below:{e}, draw=black,fill,circle,inner sep=0pt,minimum size=3pt] (e2) at (t2.corner 3) {};

\node[draw=none,minimum size=1cm,regular polygon,regular polygon sides=3] (t3) at (0.7,-1.7) {};
\node[label=left:{v}, draw=black,fill,circle,inner sep=0pt,minimum size=3pt] (v3) at (t3.corner 1) {};
\node[label=below:{c}, draw=black,fill,circle,inner sep=0pt,minimum size=3pt] (c3) at (t3.corner 2) {};
\node[label=below:{d}, draw=black,fill,circle,inner sep=0pt,minimum size=3pt] (d3) at (t3.corner 3) {};

\node[label=left:{c}, draw=black,fill,circle,inner sep=0pt,minimum size=3pt] (c2) at (-1.5,-1.2) {};
\node[label=right:{e}, draw=black,fill,circle,inner sep=0pt,minimum size=3pt] (e3) at (1.5,-1.2) {};

\node[label=left:{d}, draw=black,fill,circle,inner sep=0pt,minimum size=3pt] (d4) at (-2,-1.95) {};
\node[label=left:{e}, draw=black,fill,circle,inner sep=0pt,minimum size=3pt] (e4) at (-2,-2.7) {};

\node[label=right:{d}, draw=black,fill,circle,inner sep=0pt,minimum size=3pt] (d5) at (2,-1.95) {};
\node[label=right:{c}, draw=black,fill,circle,inner sep=0pt,minimum size=3pt] (c5) at (2,-2.7) {};

\draw (s1) -- (a1);
\draw (s1) -- (b1);
\draw (b1) -- (c2);
\draw (b1) -- (v2);
%\draw (v2) -- (c2);

 \draw (a1) -- (v3);
\draw (a1) -- (e3);
%\draw (v3) -- (e3);
\draw (v3) -- (c3);
\draw (v3) -- (d3);

\draw (v2) -- (d2);
\draw (v2) -- (e2);

\draw (c2) -- (d4);
\draw (d4) -- (e4);

\draw (e3) -- (d5);
\draw (d5) -- (c5);

\end{tikzpicture}
\\ %\hline

 & $T^\angle_{\{a,b\}}(W_6)$ \\

\begin{tikzpicture}
% create the node
\node[draw=black,minimum size=3cm,regular polygon,regular polygon sides=5] (p) {};

\node[label={a}, draw=black,fill,circle,inner sep=0pt,minimum size=3pt] (a) at (p.corner 1) {};

\node[label={b}, draw=black,fill,circle,inner sep=0pt,minimum size=3pt] (b) at (p.corner 2) {};

\node[label=below:{c}, draw=black,fill,circle,inner sep=0pt,minimum size=3pt] (c) at (p.corner 3) {};

\node[label=below:{d}, draw=black,fill,circle,inner sep=0pt,minimum size=3pt] (d) at (p.corner 4) {};

\node[label={e}, draw=black,fill,circle,inner sep=0pt,minimum size=3pt] (e) at (p.corner 5) {};

\node[label=below:{v}, draw=black,fill,circle,inner sep=0pt,minimum size=3pt] (v) at (0,0) {};
% draw a black dot in each vertex
% \foreach \x in {1,2,...,5}
%   \fill (a.corner \x) circle[radius=4pt];

\draw (a) -- (v);
\draw (b) -- (v);
\draw (c) -- (v);
\draw (d) -- (v);
\draw (e) -- (v);
\end{tikzpicture}
 & 
\begin{tikzpicture}
% create the node
\node[draw=none,minimum size=2cm,regular polygon,regular polygon sides=3] (t1) {};
%\node[label={*}, draw=black,fill,circle,inner sep=0pt,minimum size=3pt] (s1) at (t1.corner 1) {};
\node[label=left:{b}, draw=black,fill,circle,inner sep=0pt,minimum size=3pt] (b1) at (t1.corner 2) {};
\node[label=right:{a}, draw=black,fill,circle,inner sep=0pt,minimum size=3pt] (a1) at (t1.corner 3) {};

\node[draw=black,minimum size=1cm,regular polygon,regular polygon sides=3] (t2) at (-0.7,-1.7) {};
\node[label=right:{v}, draw=black,fill,circle,inner sep=0pt,minimum size=3pt] (v2) at (t2.corner 1) {};
\node[label=below:{d}, draw=black,fill,circle,inner sep=0pt,minimum size=3pt] (d2) at (t2.corner 2) {};
\node[label=below:{e}, draw=black,fill,circle,inner sep=0pt,minimum size=3pt] (e2) at (t2.corner 3) {};

\node[draw=black,minimum size=1cm,regular polygon,regular polygon sides=3] (t3) at (0.7,-1.7) {};
\node[label=left:{v}, draw=black,fill,circle,inner sep=0pt,minimum size=3pt] (v3) at (t3.corner 1) {};
\node[label=below:{c}, draw=black,fill,circle,inner sep=0pt,minimum size=3pt] (c3) at (t3.corner 2) {};
\node[label=below:{d}, draw=black,fill,circle,inner sep=0pt,minimum size=3pt] (d3) at (t3.corner 3) {};

\node[label=left:{c}, draw=black,fill,circle,inner sep=0pt,minimum size=3pt] (c2) at (-1.5,-1.2) {};
\node[label=right:{e}, draw=black,fill,circle,inner sep=0pt,minimum size=3pt] (e3) at (1.5,-1.2) {};

\node[label=left:{d}, draw=black,fill,circle,inner sep=0pt,minimum size=3pt] (d4) at (-2,-1.95) {};
\node[label=left:{e}, draw=black,fill,circle,inner sep=0pt,minimum size=3pt] (e4) at (-2,-2.7) {};

\node[label=right:{d}, draw=black,fill,circle,inner sep=0pt,minimum size=3pt] (d5) at (2,-1.95) {};
\node[label=right:{c}, draw=black,fill,circle,inner sep=0pt,minimum size=3pt] (c5) at (2,-2.7) {};

\draw (b1) -- (a1);
\draw (b1) -- (c2);
\draw (b1) -- (v2);
\draw (v2) -- (c2);

\draw (a1) -- (v3);
\draw (a1) -- (e3);
\draw (v3) -- (e3);

\draw (c2) -- (d4);
\draw (d4) -- (e4);

\draw (e3) -- (d5);
\draw (d5) -- (c5);

\end{tikzpicture}

\\ %\hline
$W_6$ & $T^\boxtimes_{\{a,b\}}(W_6)$ \\
    \end{tabular}
    \caption{An example of a graph, its induced path tree and simplicial clique tree. $W_6$ is not a line graph.}
    \label{fig:example2}
\end{figure}

There are two comments to be made about this diagram. First, we can consider the induced path tree operator as some sort of ``inverse'' or ``adjoint'' to the line graph operator. In fact, for $G \in \{\text{trees}\}$ (resp. $G \in \{\text{simpl. block graphs}\}$) we have that $T^\angle_K$ is the left (resp. right) inverse of $L$.

Second, consider the outer rectangle of the diagram. We see that the line graph operator ``passes'' the path tree operator to the clique tree operator. So, if Godsil's divisibility relation can be shown to hold between a simplicial graph and its clique tree, we will be able to derive the same relation between a graph and its path tree as a corollary. (The corollary will actually be for the \emph{edge} matching polynomial. A simple argument then gives the result for the vertex matching polynomial, as we will see below.)

We now generalize Godsil's theorem.

\begin{theorem}
Let $K$ be a simplicial clique of the simplicial graph $G = (V,E)$, and let $T \equiv T^\boxtimes_K(G)$ be the clique tree of $G$ with respect to $K$. Further, let $I(T) \equiv I^{\phi_K}(T)$ denote the relative independence polynomial. We then have the following.
\[
\frac{I(G)}{I(G \setminus K)} = \frac{I(T)}{I(T \setminus K)}
\]
\end{theorem}
\begin{proof}
We induct on $|V(G)|$. Note that if $G$ is a simplicial block graph, then $T = G$, and so the result is true. 

For the general case we get:
\begin{align*}
\frac{I(G)}{I(G \setminus K)} &= \frac{I(G \setminus K) + \sum_{v \in K} x_v I(G \setminus N[v])}{I(G \setminus K)} \\
&= 1 + \sum_{v \in K} \frac{x_v I(T^\boxtimes_{N[v]}(G \setminus K) \setminus N[v])}{I(T^\boxtimes_{N[v]}(G \setminus K))} \\
&= 1 + \sum_{v \in K} \frac{x_v I(T^\boxtimes_{N[v]}(G \setminus K) \setminus N[v]) \prod_{w \in K, w \neq v}I(T^\boxtimes_{N[w]}(G \setminus K))}{I(T^\boxtimes_K(G) \setminus K)} \\
&= 1 + \sum_{v \in K} \frac{x_v I(T^\boxtimes_K(G) \setminus N[v])}{I(T^\boxtimes_K(G) \setminus K)} \\
&= \frac{I(T^\boxtimes_K(G) \setminus K) + \sum_{v \in K} x_v I(T^\boxtimes_K(G) \setminus N[v])}{I(T^\boxtimes_K(G) \setminus K)} \\
&= \frac{I(T^\boxtimes_K(G))}{I(T^\boxtimes_K(G) \setminus K)}
\end{align*}

In the above we use the recursion formula for the independence polynomial expanding at a clique and the fact that $N[v]$ is a simplicial clique in $G \setminus K$ when $K$ is a simplicial clique. Notice also that the relative independence polynomial $I \equiv I^{\phi_K}$ is needed in order for the last equality to hold.
\end{proof}

\begin{remark}
We compute the independence polynomials of the appropriate graphs from Figure $\ref{fig:example1}$ to illustrate the divisibility relations proved in the preceding theorem: $I(L(P),x) = 1 + x_s + x_y + x_z + x_w + x_sx_w$, $I(T_{\{s\}}^\boxtimes(L(P))) = (1 + x_s + x_y + x_z + x_w + x_sx_w) \cdot (1 + x_w) = I(L(P),x) \cdot (1 + x_w)$
\end{remark}

The proof we gave for the previous theorem is essentially the one Godsil gives for his original theorem, except that we deal with simplicial cliques rather than vertices. The previous theorem now yields the following corollaries.

\begin{corollary}
$I(G)$ divides $I(T^\boxtimes_K(G))$ for any simplicial graph $G$ with simplicial clique $K$.
\end{corollary}
\begin{proof}
We have seen that $G \setminus K$ is a simplicial graph. The previous theorem can be written as:
$$\frac{I(T^\boxtimes_K(G))}{I(G)} = \frac{I(T^\boxtimes_K(G) \setminus K)}{I(G \setminus K)} = \frac{\prod_{v \in K} I(T^\boxtimes_{N[v]\setminus K}(G\setminus K))}{I(G \setminus K)}$$
Then since $N[v] \setminus K$ is a simplicial clique in $G \setminus K$, by induction we have the denominator divides any term in the numerator, so the right hand side is a polynomial, as desired.
\end{proof}

\begin{corollary}
Given a simplicial graph $G$, we have that $\lambda_1(G) \leq \frac{-1}{4(\omega - 1)}$.
\end{corollary}
\begin{proof}
By the previous corollary we have $\lambda_1(G) \leq \lambda_1(T^\boxtimes_K(G))$. Then by the commutativity of the diagram, we have seen $T^\boxtimes_K(G) = L(T^\angle_K(G))$. Hence we have $\lambda_1(T^\boxtimes_K(G)) \leq \frac{-1}{4(\omega - 1)}$ is equivalent to the identical root bound on $\mu_E(T^\angle_K(G))$. Godsil provides bounds on this root by relating the matching polynomial of a tree to its characteristic polynomial, and then bounding the roots of the characteristic polynomial by its maximal degree $\Delta$. Since the maximum degree of the vertices in $T^\angle_K(G)$ is $\omega$, we get our desired bound.
\end{proof}

% (12/22): moved from above and made more clear
\begin{remark}
In their original paper, Heilmann and Lieb prove a root bound for weighted matching polynomials, where one puts weights on the vertices. Since the previous corollary works in the multivariate case, one could use this framework to derive similar results for weighted independence polynomials.
\end{remark}

\subsection{Other Bound on $\lambda_1$}

Briefly we mention some easy lower bounds on $\lambda_1(G)$. In what follows we let $G$ be any graph. First we note how modifying our graph by removing edges or removing vertices affects $\lambda_1(G)$.

\begin{proposition} \label{prop:subgraph_rootbound}
Let $G$ be any graph, $v$ a vertex in that graph, and $e = \{u,w\}$ an edge in the graph.
\begin{enumerate}
    \item $\lambda_1(G \setminus v) \leq \lambda_1(G)$
    \item $\lambda_1(G \setminus e) \leq \lambda_1(G)$
\end{enumerate}
\end{proposition}
\begin{proof}
To prove these we need the following recurrences:
$$I(G) = I(G \setminus v) + x I(G \setminus N[v])$$
$$I(G) = I(G \setminus e) - x^2 I(G \setminus (N[u] \cup N[w]))$$

To prove the first statement we prove the following statement by induction: Given any $H \subset V(G)$, we have $I(G \setminus H)$ is nonnegative on the interval $[\lambda_1(G), \infty)$. If $G \setminus H$ is not the empty graph, then $I(G \setminus H)$ is not the zero polynomial so this implies that $\lambda_1(G \setminus H) \leq \lambda_1(G)$. If $G \setminus H$ is the empty graph it is trivially true. 

For $|V(G)| = 1$, it is easily checked to be true. Assuming this to be true for $|V(G)| \leq n - 1$, let $G$ be a graph with $|V(G)| = n$. Then if $H = G$, we noted this is trivially true. Then it suffices to show that $\lambda_1(G \setminus v) \leq \lambda_1(G)$. By induction we know $I(G \setminus N[v])$ is nonnegative on $[\lambda_1(G \setminus v), \infty)$. Then we know $x I(G \setminus N[v])$ is nonpositive on $[\lambda_1(G \setminus v), 0)$ (all the roots of independence polynomials are negative).  By the recurrence relation, $I(G)$ at $\lambda_1(G \setminus v)$ is nonpositive, so by the intermediate value theorem $I(G)$ has a root in $[\lambda_1(G \setminus v), 0)$, as desired.

To prove the second claim, since $G \setminus (N[u] \cup N[w])$ is a induced subgraph of $G \setminus e$, we have $I(G \setminus (N[u] \cup N[w]))$ is nonnegative on $[\lambda_1(G \setminus e), \infty)$. By the recurrence, we have that $I(G)$ evaluated at $\lambda_1(G \setminus e)$ is nonpositive, and so by the intermediate value theorem we see $\lambda_1(G \setminus e) \leq \lambda_1(G)$.
\end{proof}

Using this we can get the following simple lower bound on $\lambda_1$:

\begin{proposition}
$\frac{-1}{\omega} \leq \lambda_1(G)$
\end{proposition}
\begin{proof}
Let $K_\omega \leq G$ be the largest clique in $G$. Then by our previous proposition we have $\lambda_1(K_\omega) \leq \lambda_1(G)$. We have $I(K_\omega) = 1 + \omega x$, so $\lambda_1(K_\omega) = \frac{-1}{\omega}$.
\end{proof}

These results hold for all graphs, but combining these with our previous results for simplicial graphs $G$, we see:
$$\frac{-1}{\omega} \leq \lambda_1(G) \leq \frac{-1}{4(\omega-1)}$$

\section{Failure of the Root Bounds}

Recall we have the following inclusions of types of graphs:
\[
\{ \text{Line Graphs} \} \subset \{ \text{Simplicial Graphs} \} \subset \{\text{Claw-Free Graphs}\}
\]

The root bounds for the matching polynomial carry over to the independence polynomial for line graphs. And by extending the proof method of Godsil, we demonstrated the equivalent root bounds for simplicial graphs. The natural next question is: how general can the graphs get before the root bound fails?

In what follows we provide a claw-free graph (which is not simplicial) for which the root bound fails. We then provide a much weaker root bound for claw-free graphs. It is unknown whether this weaker root bound is tight due to our lack of examples of claw-free graphs which are not simplicial.

\subsection{Schläfli Graph}

The Schläfli graph is the unique strongly regular graph with parameters 27, 16, 10, 8. It is the complement of the Clebsch graph, the intersection graph of the 27 lines on a cubic surface. The Clebsch graph is triangle free, and hence the Schläfli graph is claw-free. We refer the reader to \cite{schlaf} for a comprehensive reference on the Schläfli graph and related graphs.

Keeping in mind that our root bound is equivalent to the statement $\lambda_1(G) \cdot 4 \cdot (\omega-1) \leq -1$, we calculate the following.

\begin{lemma}
We have the following:
\begin{enumerate}[(i)]
    \item The independence polynomial of the Schläfli graph is $45t^3 + 135t^2 + 27t + 1$.
    \item The clique size of the Schläfli graph is 6.
    \item $\lambda_1(\text{Schläfli graph}) \cdot 4 \cdot (\omega -1) > -1$
\end{enumerate}
\end{lemma}
\begin{proof}
One can calculate the independence polynomial and clique size using any computer algebra system; we used Sage.

To show our graph breaks the root bound it suffices to show that $I(G)(t/20)$ has a root in $(-1,0)$. In fact we can easily calculate that $I(G)(-1/20) = -29/1600$ while $I(G)(0) = 1$, so there is a root in $(-1, 0)$.
\end{proof}

\subsection{Weaker Root Bounds for Claw-free Graphs}
Given any claw-free graph $G$, we can introduce a simplicial clique by modifying the graph as follows:

\begin{lemma}
    Let $G$ be a claw-free graph. Given any vertex $v \in G$, we can form a new graph $S_v(G)$ by connecting all of $N[v]$ together to form a clique. Then, $S_v(G)$ is claw-free and $\{v\}$ is a simplicial clique in $S_v(G)$.
\end{lemma}
\begin{proof}
It is clear that $\{v\}$ will be a simplicial clique in $S_v(G)$. To see that $S_v(G)$ is claw-free, suppose one of the added edges creates a claw. Then we have $u,w \in N(v)$ and a claw with some $u$ as the internal node and $w$ as a leaf. Since we have connected all of the neighbors of $v$ together, we must have the other two leaves of the claw outside of $N[v]$. However these two vertices therefore are not connected to $v$ or each other, and hence form a claw with $u$ as the internal node and $v$ as the other leaf. This provides a contradiction since $G$ is claw-free.
\end{proof}

When analyzing the clique tree of $S_v(G)$ starting at the newly formed simplicial clique $\{v\}$, we notice that the first rung of the clique tree is $\{v\}$, the second rung is $N(v)$, and beyond that are clique trees that live in $G \setminus N[v]$. This observation immediately yields the following:

\begin{proposition}
Given any claw-free graph $G$ and a vertex $v \in G$, we have:
$$\lambda_1(S_v(G)) \leq \frac{-1}{4 \cdot \max\{\omega - 1, \text{deg}(v)\}}$$
This yields the following root bound for $G$:
$$\lambda_1(G) \leq \frac{-1}{4 \cdot \max\{\omega - 1, \delta\}}$$
\end{proposition}
\begin{proof}
By Proposition \ref{prop:subgraph_rootbound}, we have $\lambda_1(G) \leq \lambda_1(S_v(G))$. To optimize the bound we pick the vertex $v$ which has minimal degree in the graph, $\delta$. 
\end{proof}

In the Schläfli graph we have a large gap between the clique size of $6$ and minimal degree of $16$. We think that other non-simplicial claw-free graphs with a large gap between clique size and minimal degree may provide good candidates for studying this root bound. Further, finding a family of graphs which require this looser bound could assist in showing how optimal this bound is for non-simplicial claw-free graphs.

\section{Other Remarks}

Above, we presented independence polynomials analogues to the real-rootedness (subsequently real stability) and the root bounds of the matching polynomial. We expect other results about the matching polynomial to be generalizable to the independence polynomial. In what follows we list a few examples and comment on these.

In \cite{rootbounds}, Fisher and Solow remark that $I(G)^{-1}$ can be viewed as a generating function which enumerates the number of $n$ letter words, where the letters are the vertices of the graph and two letters commute iff they have an edge between them on the graph. Similarly in \cite{godsil}, Godsil shows that $\frac{x (x^{2n}\mu_V(G,x^{-1}))'}{x^{2n}\mu_V(G,x^{-1})}$ is a generating function in $x^{-1}$ for closed tree-like walks in $G$. We believe that there is a multivariate generalization of Fisher and Solow's remark by working in the ring $\mathbb{Z}[x_1, \ldots, x_n]$ where variables commute if and only if they correspond to vertices in the graph $G$ which share an edge. Godsil's tree-like result should be a combinatorial consequence of the more general Fisher and Solow result.

In a previous paper of Bencs, Christoffel-Darboux like identities are established for the independence polynomial \cite{CDidentities}. One can similarly establish multivariate generalizations of these identities. By generalizing in this way, one can give a single identity that implies all the others through simple multivariate operations.

Another area of interest is studying independent sets in hypergraphs. One can naturally define the multivariate independence polynomial of a hypergraph. Namely given a hyper graph $G = (V, E)$ a set $S \subset V$ is independent if $e \not\subset S$ for all edges $e \in E$. If two edges are comparible in $G$ ($e \subset f$), then we note that by removing $f$ from the edge set we do not change the independent sets of $G$. If $G$ contains any edges of size one, then that vertex never shows up in the independence polynomial so we can further reduce $G$ by removing that vertex. Thus we can do this to obtain the reduction, $\tilde{G}$, of $G$ which has the same multivariate independence polynomial and has no comparable edges and no edges of size $1$.
\begin{proposition}
Given a hypergraph $G$, $I(G,x)$ is same-phase stable if and only if $\tilde{G}$ is a $2$-uniform claw-free graph.
\end{proposition}
\begin{proof}
As noted, $I(G,x) = I(\tilde{G},x)$, so if $\tilde{G}$ is $2$-uniform and claw-free we see $I(G,x)$ is same-phase stable by previous results. If $\tilde{G}$ is not $2$-uniform, then we have some edge $e$ with $|e| > 2$. If $I(\tilde{G},x)$ were same-phase stable then we could restrict to the subgraph of vertices in $e$ and obtain a same-phase stable independence polynomial. Since no other edges are comparable to $e$ by construction of $\tilde{G}$, we have this subgraph only contains the edge $e$. Then we can diagonalize to get the independence polynomial $(1+x)^n - x^n$. If $I(\tilde{G},x)$ were same-phase stable, this polynomial would be real rooted. However this would imply that its derivatives were real rooted, namely $(1+x)^3 - x^3 = 1 + 3x + 3x^2$ would be real rooted, a contradiction.
\end{proof}

\bibliographystyle{amsalpha}
\bibliography{indep_poly}

\end{document}